\documentclass[11pt]{article}
\usepackage{url}

\usepackage[utf8]{inputenc}

\usepackage{geometry}
\providecommand{\keywords}[1]
{
  \small    
  \textbf{\textit{Keywords---}} #1
}
\usepackage[utf8]{inputenc}
\usepackage[english]{babel}
\usepackage{amsmath,amssymb,mathtools,bm,etoolbox}
\usepackage{amsthm}
\usepackage{fancyhdr}
\usepackage{graphicx}
\usepackage{makeidx}
\usepackage{mathrsfs}
\usepackage{showidx}
\usepackage{setspace}
\usepackage{dsfont}
\usepackage{enumitem}
\usepackage{dirtytalk}
\usepackage{url}
\numberwithin{equation}{section}

\newcommand{\1}[1]{\mathds{1}_{\{#1\}}}

\newcommand{\A}{\mathcal{A}}

\newcommand{\As}{\hat{\mathcal{A}}}
\newcommand{\F}{{\mathscr{F}}}
\newcommand{\Fs}{{\hat{\mathscr{F}}}}
\newcommand{\Zd}{{\mathbb{Z}^{d}}}

\newcommand{\E}{\mathbb{E}}

\newtheorem{theorem}{Theorem}[section]
\newtheorem{conjecture}[theorem]{Conjecture}

\newtheorem{lemma}[theorem]{Lemma}
\newtheorem{prop}[theorem]{Proposition}
\theoremstyle{definition}
\newtheorem{definition}[theorem]{Definition}
\newtheorem{assumption}[theorem]{Assumption}
\newtheorem{remark}[theorem]{Remark}

\newcommand{\threepartdef}[6]
{
    \left\{
        \begin{array}{lll}
            #1 & \mbox{ if } #2 \\
            #3 & \mbox{ if } #4 \\
            #5 & \mbox{  } #6
        \end{array}
    \right.
}
\newcommand{\fourpartdef}[8]
{
    \left\{
        \begin{array}{lll}
            #1 & \mbox{ if } #2 \\
            #3 & \mbox{ if } #4 \\
            #5 & \mbox{ if } #6 \\
            #7 & \mbox{  } #8
        \end{array}
    \right.
}
\usepackage{authblk}
\begin{document}
\title{\Large{General Contact Process with Rapid Stirring} }
\author{Segev Shlomov}
\author{Leonid  Mytnik *}
\affil{
segevs@campus.technion.ac.il,    leonidm@technion.ac.il \\ \vspace{1ex}

\small{Faculty of Industrial Engineering and Management
Technion — Israel Institute of Technology,
Haifa 32000, Israel}}

\maketitle

\begin{abstract}
We study the limiting behavior of an interacting particle
system evolving on the lattice $\Zd$ for $d\ge 3$. The model is known as the contact process with
rapid stirring. 
The process starts with a single particle at the origin. Each particle may die, jump to a neighboring site if it is vacant or split. In the case of splitting, one of the offspring takes the place of the parent while the other, the newborn particle, is sent to another site in $\Zd$ according to a certain distribution; if the newborn particle lands on an occupied site, its birth is suppressed.
We study the asymptotic behavior of the critical branching rate as the jumping rate (also known as the stirring rate) approaches infinity. 
\end{abstract}

\keywords{Contact processes, Rapid stirring, Interacting particle systems, Asymptotic behavior}

\textbf{AMS MSC 2010:} Primary 82C22, 60K35

\let\thefootnote\relax\footnotetext{*Supported by the Israel Science Foundation Grant 1704/18.}

\section{Introduction}

The basic contact process with rapid stirring has state space $S=\{0,1\}^{\Zd}$ where $1$ at site $x\in \Zd$ means that site $x$ is occupied by a particle and $0$ means that the site is empty.
Each particle dies with rate $1$ and splits at rate $\lambda$. If a split occurs at site $x$ then one of the children replaces its parent at site $x$ while the other child is sent to a site $y$ chosen uniformly within the nearest neighboring sites of $x$. In addition, each pair of neighboring states values exchanges at rate $N$ (stirring).

The asymptotic behavior of this and related processes was studied by Durrett and Neuhauser in \cite{durrett1994particle}. 
They obtained results about the existence of phase transition, for various systems, when the stirring rate goes to infinity.
Let $\lambda_{c}$ denotes the minimal branching rate for which the process survives with positive probability.
For the basic contact process, starting with a single particle at the origin, Durrett and Neuhauser, showed that $\lambda_c^N\rightarrow 1\;$ as $N\rightarrow\infty$.

Later on, Konno in \cite{konno1995asymptotic} improved the result of Durrett and Neuhauser by getting a more detailed description of the asymptotic behavior of $\lambda_c^N$. His main result is as follows.
For $N>0$ define:
\begin{equation*}
\phi_d(N) = \threepartdef 
{ \frac{1}{N^\frac{1}{3}}} {d=1,}
{\frac{\log{N}}{N}} {d=2,}
{\frac{1}{N}} {d>2.}
\end{equation*}
Then we have
\begin{equation}\label{eq:3}
\lambda_c^N\approx 1+C_*\phi_d(N).
\end{equation}
Here $\approx$ means that if $C_*>0$ is small (large) then the right-hand side of (\ref{eq:3}) is a lower (upper) bound of the left-hand side for all large enough $N$.
After that, Katori in \cite{katori1994rigorous} got bounds on $C_*$ for dimensions $d\ge 3$. He showed:
\begin{align*}
    \frac{1}{2d(2d-1)}\le\liminf_{N\to\infty}(\lambda_{c}^{N}-1)N\le\limsup_{N\to\infty}(\lambda_{c}^{N}-1)N\le \frac{G(0,0)-1}{2d},
\end{align*}
where $G(\cdot,\cdot)$ is the Green function of the simple random walk on $\Zd$.

Later on, Berezin and Mytnik in \cite{berezin2014asymptotic} improved the results of Konno and Katori by getting more precise information on the asymptotic behavior of $\lambda_c^N$ in dimensions $d\ge 3$. They improved the lower bound on $\lambda_{c}^{N}$ and by using Katori's result in \cite{katori1994rigorous} they got the sharp asymptotic behavior of the critical $\lambda_{c}$. Their main result was:
\begin{equation}\begin{aligned}
\lim_{N\to\infty}\frac{\lambda_c^N-1}{\frac{\vartheta}{N}}=1, \;\;\;\;\;\;
\text{where }\;\vartheta=\frac{G(0,0)-1}{2d}\; .
\end{aligned}\end{equation}

Finally, Levit and Valesin in \cite{levit2017improved} studied the same model for $d=2$ and established a  lower bound for $\lambda_c^N$. Their main result for $d=2$, was
\begin{equation*}
     \liminf_{N\to\infty}\frac{\lambda_{c}^{N}-1}{\frac{\log{N}}{N}}\ge \frac{1}{4\pi}.
\end{equation*}

\noindent The aim of this paper is to study the limiting behavior of $\lambda_{c}^{N}$ for a more general case. In particular, we generalize the model as follow: the location of one of the offspring is distributed according to some symmetric, probability function $P^{b}$ which satisfies Assumption \ref{assump} below, and not according to the uniform distribution on its $2d$ closest neighboring as in the nearest neighbor contact process.\\
 
The rest of the paper is organized as follows: In Section \ref{TheModel} we present the model and state our main result. Section \ref{Proof of Main Result} is dedicated to the proof of our main result. This proof is based on two propositions which we prove in Sections \ref{mainPropSection} and \ref{vartheta}.

\section{Model and Main Result}\label{TheModel}
The contact process with rapid stirring is defined on the lattice $\Zd$. Let $N$ be a \say{large} integer parameter, which in the sequel we will send to infinity.
The state of the process at time $t$ is given by the function $\xi_t^N : \Zd \to{\{0,1\}}$, where the value of $\xi_t^N(x)$ represents the number of particles presented at $x$ at time $t$ (it equals zero or one).
Starting with a single particle at the origin, the basic rules for the evolution of the process are:

\begin{enumerate}

  \item particles die at rate $1$ without producing offspring;
  \item particles split into two at rate $\lambda=1 + C(N)$, where $C(N)\xrightarrow{N\rightarrow\infty}0$. If a split occurs at $x\in{\Zd}$, then one of the offspring replaces the parent, while the other is sent to a site $y$ with probability $P^{b}(x,y)$. If a newborn particle lands on an occupied site, its birth is suppressed;
  \item for each $x,y\in{\Zd}$ such that $||x-y||_1=1$ the values of $\xi^N$ at $x$ and $y$ are exchanged at rate $N$;
  \item the above mechanisms are independent.
\end{enumerate}
\begin{remark}
   We say that events occur at some rate if the times between events are independent and exponentially distributed with that rate.
\end{remark}
In the rest of this paper we will frequently use the following assumption about the probability function $P^{b}(x,y)$.
\begin{assumption}\label{assump}
Let $P^{b}(x,y)$ be a function of $||x-y||_{1}$. It means that there exists a function $f$ such that  $P^{b}(x,y)=f(||x-y||_{1})$.
\end{assumption}

Before we present the main result let us introduce the following notation:
\begin{itemize}
\item Let $|\xi_t^N|=\sum_{x\in \Zd}\xi_{t}^{N}(x)$ denote the total number of particles at time $t$.
\item Let $\Omega_{\infty}^{N}=\{|\xi_t^N|>0\;\;\;\forall t\ge0\}$ denote the event that the process survives.
\item Let $\rho_\lambda^N=P(\Omega_{\infty}^{N})$ denote the probability that the process survives.
\item Let $\lambda_c^N=\inf\{\lambda\ge0:\rho_\lambda^N>0\}$ denote the minimal branching rate such that the process survives with positive probability (also known as the critical branching rate). 
\end{itemize}

Note that $\rho_\lambda^N$ describes the probability of survival of the process. For all $\lambda$ smaller than $\lambda_c^N$ the process dies out almost surely, and for $\lambda$ greater than $\lambda_c^N$ the process survives with probability $\rho_\lambda^N$.


Our main result gives a lower bound on the critical branching rate. It is presented in the following theorem.
\begin{theorem}\label{mainTHM}
Let $d\ge 3$ and assume that $P^{b}(x,y)$ satisfies Assumption \ref{assump}. Then,
    \begin{equation*}
        \liminf_{N\to\infty}\frac{\lambda_{c}^{N}-1}{\frac{\vartheta}{N}}\ge 1,
    \end{equation*}
    where
    \begin{equation}\label{thetadef}
    \vartheta\equiv
    (2d)^{-1}\sum_{x,y\in \Zd}P^{b}(0,x)P^{b}(0,y)G(x,y).
    \end{equation}
\end{theorem}
Our main theorem shows that the asymptotic behavior of the critical rate is at least $1+\frac{\vartheta}{N}$.

\begin{conjecture}\label{conj}
We conjecture that in our general model the lower bound is tight. That is, we conjecture that $\limsup_{n\to\infty}\frac{\lambda_{c}^{N}-1}{\frac{\vartheta}{N}}=1$ and thus,
\begin{equation*}
    \lim_{n\to\infty}\frac{\lambda_{c}^{N}-1}{\frac{\vartheta}{N}}=1.
\end{equation*}
\end{conjecture}

\subsection*{Special case considered in \cite{berezin2014asymptotic}}
Let us show that our main result is consistent with the result of Berezin and Mytnik in \cite{berezin2014asymptotic} who considered a special case of our model.
In \cite{berezin2014asymptotic} the same model is considered with a particular distribution of the location of a particle after splitting: when a particle splits its child's location is distributed uniformly over its closest neighbors. That is, the branching probability function is $P^{b}(y,x)=(2d)^{-1}\1{|x-y|=1}$).
Let $\varphi_{d}=\{x\in\Zd : ||x||_{1}=1\}$ denote the set of nearest neighbours of the origin. Thus, according to (\ref{thetadef}),
\begin{align*}
    \vartheta&=
    (2d)^{-1}\sum_{x\in \varphi_{d}}\sum_{y\in \varphi_{d}}\frac{1}{2d}\frac{1}{2d}G(x,y).
\end{align*}
For any $x\in\Zd$, let $G(x,\varphi_{d})=\sum_{y\in\varphi_{d}}G(x,y) $ denote the expected number of times a simple symmetric nearest neighbour random walk visits at the set of sites $\varphi_{d}$, given that it started from site $x$. Thus,
\begin{align*}
    \vartheta&=
    (2d)^{-1}\sum_{x\in \varphi_{d}}\frac{1}{2d}\frac{1}{2d}G(x,\varphi_{d})
    =(2d)^{-2}\sum_{x\in \varphi_{d}}\frac{1}{2d}G(x,\varphi_{d}).
\end{align*}
Notice that when a simple symmetric nearest neighbour random walk jumps from the origin, then with probability $\frac{1}{2d}$ it jumps to a site $x$ in $\varphi_{d}$. Thus, by using first step analysis we get, 
\begin{align*}
    \vartheta =  (2d)^{-2}G(0,\varphi_{d})
    =\frac{G(0,0)-1}{2d}
\end{align*}
which coincides with the result obtained by Berezin and Mytnik in Theorem 3 in \cite{berezin2014asymptotic}.

Note that the lower bound in \cite{berezin2014asymptotic} is tight. This motivates our Conjecture \ref{conj}.

\section{Proof of Main Result}\label{Proof of Main Result}

Before we start the proof of the main result let us introduce an additional notation. 
      For any $x,y\in \Zd$ we say that $x,y$ are $\ell$-neighbours if $\;{||x-y||_{1}=\ell}$. In that case, we denote it by $x\underset{\ell}{\sim}y\;$.
      let $h_{\ell}^{d}=\sum_{x\in \Zd}1_{\{0\underset{\ell}{\sim}y\}}$ be the total number of $\ell$-neighbours of the origin.

Recall that in this work we focus on the probability function $P^{b}(x,y)$ that satisfies Assumption \ref{assump}, which means that it depends only on the distance (in $l_{1}$-norm) between $x$ and $y$. Thus, we may write
 \begin{equation}
     P^{b}(x,y)=\sum_{\ell}\frac{p_{\ell}}{h_{\ell}^{d}}\cdot 1_{\{x\underset{\ell}{\sim}y\}}
 \end{equation}
for some discrete distribution $p_{\ell}$ on $\mathbb{N}$.
This means that whenever a particle splits one of its offspring jumps to an $\ell$-neighbor with probability $p_{\ell}$ and its position is distributed uniformly within all the $\ell$-neighbors. We will also frequently use the following notation:
 \begin{itemize}
    \item We denote particles by Greek letters $\alpha,\beta,\gamma$.
     \item  Since we start with a single particle and each particle may split into two we use a binary vector representation of particles. For example, particle $(1)$ is the first particle, the one we start the process with; particles (1,0) and (1,1) are the two children of particle (1) and so on.
     \item For each particle $\alpha$ let $\mathscr{C}_{0}(\alpha)$ and $\mathscr{C}_{1}(\alpha)$ denote the two children of particle $\alpha$. For example, if $\alpha=(1,1)$ then $\mathscr{C}_{0}(\alpha)=(1,1,0)$ and $\mathscr{C}_{1}(\alpha)=(1,1,1)$. In addition, we say that particle $\beta$ is the child of particle $\alpha$ if $\beta=\mathscr{C}_{0}(\alpha)$ or $\beta=\mathscr{C}_{1}(\alpha)$.
    \item For each particle $\alpha$ let  $\mathscr{P}(\alpha)$ denote the parent of particle $\alpha$. In addition, we say that particle $\alpha$ is the parent of particle $\beta$ if $\mathscr{P}(\beta)=\alpha$.
 \end{itemize}

In order to prove Theorem \ref{mainTHM} we set the branching rate $\lambda=1+\frac{\theta}{N}$ and we show that for any $\theta<\vartheta$ the process dies out with probability one, which gives a lower bound for the critical branching rate. It will be more convenient to deal with the speeded-up process
\begin{equation}\label{speed-pro}
    {\hat{\xi}_{t}^{N}}={\xi}_{N t}^{N}.
\end{equation} This process obeys the same rules as ${\xi}_{t}^{N}$ with the exception that all events occur with rates multiplied by $N$. Thus, particles die with rate $N$, split with rate $\theta+N$ and perform stirring with rate $N^{2}$.\\

The next proposition is crucial for the proof of the main Theorem \ref{mainTHM}. It is proved in Section \ref{mainPropSection}.
\begin{prop}\label{mainprop}
Fix an arbitrary $\theta<\vartheta$. Then there exists $N_{\theta}>0$ such that for any $N\ge N_{\theta}$ and for any $t\ge 0$,
    \begin{equation}
    \hat{m}_{t}^{N}:=\E(|\hat{\xi}_{t}^{N}|)\le 
    \exp{\left[
    \left(\frac{1}{2}(\theta-\vartheta)\right)t+2
    \right]}.
    \end{equation}
\end{prop}
With Proposition \ref{mainprop} at hands we are ready to give the proof of Theorem \ref{mainTHM}.
\begin{proof}[Proof of Theorem \ref{mainTHM}\nopunct].
From Proposition \ref{mainprop} it follows that for every $\theta<\vartheta$, uniformly on all $N$ sufficiently large, the expected number of particles approaches zero (exponentially fast) when $t$ approaches infinity. Since the number of particles is a non-negative integer random variable then with probability one the process dies out for every  $\theta<\vartheta$. This proves Theorem  \ref{mainTHM}.
\end{proof}

\section{Proof of Proposition \ref{mainprop}}\label{mainPropSection}

In order to prove Proposition \ref{mainprop} we are going to build the differential equation for the expected number of particles. Then, we bound the expected number of $\ell$-neighbors from below by looking only at pairs of particles which are parents and children. After that, we
bound from above the total mass of the process by ignoring collisions between distant relatives. 

\subsection{Differential Equation for the Expected Number of Particles}
The first ingredient in the proof of Proposition \ref{mainprop} is a derivation of the equation for the expected number of particles. To do this we need the following notation:
\begin{itemize}
    \item Let $\A_{t}$ denote the set of all particles alive at time $t$ in the process $\xi_{t}^{N}$.
    \item Let $\F_{t}$ denote the natural filtration of the process $\xi_{t}^{N}$ until time $t$. 
    \item Let $m_{t}^{N}=\E(|\xi_{t}^{N}|)$ denote the expected number of particles at time $t$ and let ${\hat{m}_{t}^{N}=m_{Nt}^{N}}$.
    \item Let $I_{t}^{\ell,N}=\E[\sum_{x\underset{\ell}{\sim}y}\xi_{t}^{N}(x)\xi_{t}^{N}(y)]$ denote twice the expected number of $\ell$-neighbours alive at time $t$ and let $\hat{I}_{t}^{\ell,N}=I_{Nt}^{\ell,N}$.
\end{itemize}
In addition, recall that in that process $\xi_{t}^{N}$, each particle splits at rate $1+\frac{\theta}{N}$, dies at rate $1$ and jumps to each one of its closest vacant neighbouring sites with rate $N$. Since each particle has $2d$ closest neighbours the total jumping rate of a particle is $2dN$. Thus, 
\begin{equation*}
    \mu := 2+\frac{\theta}{N}+2dN
\end{equation*}
denotes the total rate of events (splitting, dying and stirring) that can occur for each alive particle.\\

Before we present the differential equation for the expected number of particles in the speeded-up process (see (\ref{speed-pro})) let us state a lemma which will be useful for the rest of the paper.
\begin{lemma}\label{finiteM}
For any finite time $t$ and any stirring rate $N$ the following holds:
 \begin{align*}
     &(1)\;\;\;\;P(|\xi_{t}^{N}|<\infty)=1 \;\;\text{and}\;\; m_{t}^{N}\le e^{\mu t}<\infty;
     \\&(2)\;\;\;\;P(|\hat{\xi}_{t}^{N}|<\infty)=1 \;\;\text{and}\;\; \hat{m}_{t}^{N}\le e^{\mu Nt}<\infty.
 \end{align*}
 \end{lemma}
The proof of this lemma is simple and is based on bounding $|\xi_{t}^{N}|$ by the Yule process. Since the proof is standard it is omitted.
The differential equation for the expected number of particles in the speeded-up process, $\{\hat{\xi}_{t}^{N}\}_{t\ge 0}$, is given in the next lemma.

\begin{lemma}\label{Differential equation}
\begin{equation}
    \hat{m}_{t}^{N}=\hat{m}_{0}^{N}+\int_{0}^{t} \theta\hat{m}_{s}^{N}ds- (1+\frac{\theta}{N})\sum_{\ell}{\frac{p_{\ell}}{h_l}\int_{0}^{t}N\hat{I}_{s}^{\ell,N}}ds.
\end{equation}
\end{lemma}
The analogous results are frequently stated as standard (see page 152 in \cite{schinazi2012classical} and page 842 in \cite{konno1995asymptotic}) so the proof of the lemma is omitted. A complete proof of the lemma can be found in \cite{mythesis}.

\subsection{Bounding the expected number of $\ell$-neighbors}
Another ingredient in the proof of Proposition 3.1 is bounding the expected number of $\ell$-neighbors.
It is clear from Lemma \ref{Differential equation} that in order to bound from above the expectation of the total mass of the speeded-up process $\hat{\xi}_{t}^{N}$, we need to bound from below the expected number of $\ell$-neighbors for every $\ell$. The main idea that allows us to get a useful lower bound on this quantity at time $s$ is based on counting only particles that are $\ell$-neighbors at time $s$ and have a common parent that splits at some time in $[s-\tau_{N},s)$, where
\begin{equation}\label{tauN}
    \tau_{N}=\frac{\ln(N)}{N^{2}}.
\end{equation}
The reason for such a definition of $\tau_{N}$ is that two particles need the order of $N^{-2}$ units of time to travel \say{long} distance so that a collision between them, after this time, happens with very small probability.
This explains why we choose $\tau_{N}$ a little bit bigger than $N^{-2}$.
A similar technique is used in a number of papers (see \cite{berezin2014asymptotic},\cite{konno1995asymptotic},\cite{durrett1999rescaled}).
In Lemmas \ref{bound I} -- \ref{big} below, we will use this idea in order to bound from below the expected number of $\ell$-neighbors.

Before we proceed, we need an additional notation:
\begin{itemize}
    \item Let $\As(t)$ denote the set of all particles alive at time $t$ in the process $\hat{\xi}_{t}^{N}$.
    \item Let $\Fs_{t}$ denote the natural filtration of the process $\hat{\xi}_{t}^{N}$ until time $t$. 
    \item Let ${T}_{\alpha}$ denote the time at which particle $\alpha$ dies or splits in the process $\hat{\xi}_{t}^{N}$.
    \item Let $B_{t}^{\alpha}$ denote the position (on the lattice) of particle $\alpha$ at time $t$ in the process $\hat{\xi}_{t}^{N}$. If particle $\alpha$ is not alive at time $t$ (dead or not born yet) then we write $B_{t}^{\alpha}=\bigtriangleup$.
    \item Let $Z_{\alpha}^{\ell}(t)=1(T_\alpha\in [t,t+\tau_{N}], B_{t+\tau_{N}}^{\mathscr{C}_{1}(\alpha)}\underset{\ell}{\sim}B_{t+\tau_{N}}^{\mathscr{C}_{0}(\alpha)},{T}_{\mathscr{C}_{1}(\alpha)}\ge t+\tau_{N},{T}_{\mathscr{C}_{0}(\alpha)}\ge t+\tau_{N})$ denote the indicator of the event that the two children of particle $\alpha$ created during the time interval $[t,t+\tau_{N}]$ are $\ell$-neighbours at time $t+\tau_{N}$ and let $Z_{1}^{\ell}=Z_{1}^{\ell}(0)$.
    \item Let $\zeta_{\alpha}(t)$ denote the indicator of the event that one of $\alpha$'s children created during the time interval $[t,t+\tau_{N}]$ died at the time of its birth $T_{\alpha}$ as a result of collision.
\end{itemize}

\begin{lemma}\label{bound I} For any $s\ge\tau_{N}$ and for any $\ell=1,2,...$
\begin{equation*}
    \hat{I}_{s}^{\ell,N} \ge 2\hat{m}_{s-\tau_{N}}^{N}\E[Z_{1}^{\ell}]-2\E\left[\sum_{\alpha\in \As(s-\tau_{N})}\zeta_{\alpha}(s-\tau_{N})\right].
\end{equation*}
\end{lemma}

\begin{proof}Fix arbitrary $\ell$ and recall the definition of $B_{s}^{\beta}$. Then, $\hat{I}_{s}^{\ell,N}$ can be written as follows:
\begin{equation*}
    \hat{I}_{s}^{\ell,N}=\E\left[\sum_{\alpha,\beta\in \As(s)}{1_{\{B_{s}^{\beta}\underset{\ell}{\sim}B_{s}^{\alpha}\}}}\right].
\end{equation*}
We  bound the right-hand side from below by summing only pairs of particle that share a common parent who splits during the time interval $[s-\tau_{N},s]$. Thus,
\begin{equation*}
    \hat{I}_{s}^{\ell,N}\ge
    2\E\left[\sum_{\alpha\in \As(s-\tau_{N})} {Z_{\alpha}^{\ell}(s-\tau_{N}})\right]\ge 
    2\E\left[\sum_{\alpha\in \As(s-\tau_{N})}
    \E\left[Z_{\alpha}^{\ell}(s-\tau_{N})\mid \Fs_{s-\tau_{N}}\right]\right].
\end{equation*} 
Note that given particle $\alpha$ alive at time $s-\tau_{N}$ and given $\Fs_{s-\tau_{N}}$ we have for $s\ge\tau_{N}$
\begin{equation}\label{dsfhg}
    \E[Z_{\alpha}^{\ell}(s-\tau_{N})\mid \Fs_{s-\tau_{N}}]+\E[\zeta_{\alpha}(s-\tau_{N})\mid \Fs_{s-\tau_{N}}]\ge \E[Z_{1}^{\ell}].
\end{equation}
Thus,
\begin{align*}
    \hat{I}_{s}^{\ell,N}&\ge
    2\E\left[\sum_{\alpha\in \As(s-\tau_{N})}
    (\E[Z_{1}^{\ell}]-E[\zeta_{\alpha}(s-\tau_{N})\mid \Fs_{s-\tau_{N}}])\right]
    =
    2\hat{m}_{s-\tau_{N}}^{N}\E[Z_{1}^{\ell}]-2\E\left[\sum_{\alpha\in \As(s-\tau_{N})}\zeta_{\alpha}(s-\tau_{N})\right] .
\end{align*}
\end{proof}


Now we need to get an upper bound for $\int_{\tau_{N}}^{t} N\E\left[\sum_{\alpha\in \As(s-\tau_{N})}\zeta_{\alpha}(s-\tau_{N})\right]ds$. This is done in the next lemma.
\begin{lemma}\label{bound zeta integral}
For every $N>N_{2}=\max \{\theta,4\}$ and $t\ge\tau_{N}$
\begin{equation*}
  \int_{\tau_{N}}^{t} N\E\left[\sum_{\alpha\in \As(s-\tau_{N})}\zeta_{\alpha}(s-\tau_{N})\right]ds
  \le
  \sum_{\ell}\frac{p_{\ell}}{h_{\ell}}4N\tau_{N}\int_{0}^{t} N\hat{I}_{r}^{\ell,N} dr.
\end{equation*}
\end{lemma}

\begin{proof}
In order to prove this lemma we first need to bound the term inside the integral. That is,
\begin{equation}\label{zeta}
    \Xi=\E\left[\sum_{\alpha\in \As(s-\tau_{N})}\zeta_{\alpha}(s-\tau_{N})\right].
\end{equation}
To this end, we introduce the following notation:
Let $T_{\alpha}^{c}$ (resp. $T_{\alpha}^{d}$) denote the potential split (resp. death) time of particle $\alpha$; note that $T_{\alpha}=\min(T_{\alpha}^{c},T_{\alpha}^{d})$.
Let $J_{\alpha}$ be the event that one of the particle $\alpha$'s children died at the time of its birth as a result of a collision with another particle. In addition, notice that from the memoryless properties of the exponential distribution we can deduce the following property of $T_{\alpha}^{c}$:\\
\textbf{(P1)} For any time $t\ge 0$, given $\{\Fs_{t},\alpha\in\As(t)\}$, the random variable $(T_{\alpha}^{c}-t)$ has an exponential distribution with rate $\lambda_{b}:=N+\theta$.

We start by bounding $\Xi$. By the law of total expectation, the linearity of the expectation, and the definition of $\zeta_{\alpha}(s-\tau_{N})$ we have
\begin{align*}
    \Xi  &=
    \E\left(\sum_{\alpha\in \As(s-\tau_{N})}\E\left[\zeta_{\alpha}(s-\tau_{N}) \mid \Fs_{s-\tau_{N}}\right]\right)&
\\&=
    \E\left(\sum_{\alpha\in \As(s-\tau_{N})} P\left[T_{\alpha}^{c}\in (s-\tau_{N},s) \;,\; T_{\alpha}^{d} > T_{\alpha}^{c} \;,J_{\alpha} \mid \Fs_{s-\tau_{N}},T_{\alpha}^{c}>s-\tau_{N}\right]\right).
\end{align*}
Apply \textbf{(P1)}, and then the tower property to get
\begin{align*}
    \Xi&\le
    \E\left(\sum_{\alpha\in \As(s-\tau_{N})} 
    \int_{s-\tau_{N}}^{s}\lambda_{b}   \; \E(1_{\{J_{\alpha}\}}\mid \Fs_{s-\tau_{N}},T_{\alpha}^{c}=t,T_{\alpha}^{d}>t)dt \right)
    \\&\le
    \E\left(\sum_{\alpha\in \As(s-\tau_{N})} 
    \int_{s-\tau_{N}}^{s}\lambda_{b}   \; \E[\E(1_{\{J_{\alpha}\}}\mid F_{t-},T_{\alpha}^{c}=t,T_{\alpha}^{d}\ge t)\mid \Fs_{s-\tau_{N}},T_{\alpha}^{c}=t,T_{\alpha}^{d}\ge t] dt \right).
\end{align*}
Recall that ${n_{\alpha}^{\ell}(t)}$ denotes the number of $\ell$-neighbours of particle $\alpha$ alive at time $t$. In addition, when a particle splits one of its offspring is sent to an $\ell$-neighbour with probability $p_{\ell}$ and is distributed uniformly within all the $\ell$-neighbours. Thus, since the filtration until time $t-$ is given and particle $\alpha$ splits at time $t$, then the conditional probability that its child lands on an occupied site, given its child is sent to an $\ell$-neighbour, is equal to $\frac{n_{\alpha}^{\ell}(t)}{h_{\ell}}$. Therefore, we get
  \begin{align*}
   \Xi\le
    \E\left(\sum_{\alpha\in \As(s-\tau_{N})} 
    \int_{s-\tau_{N}}^{s}\lambda_{b}  \sum_{\ell}\frac{p_{\ell}}{h_{\ell}}\E[{n_{\alpha}^{\ell}(t)} \mid \Fs_{s-\tau_{N}},T_{\alpha}^{c}\ge t,T_{\alpha}^{d}\ge t]dt \right).
\end{align*}
Notice that for each $\ell$ and a particle $\alpha$ alive at any time $t$,
 \begin{align*}
 \E[n_{\alpha}^{\ell}(t)\mid \Fs_{s-\tau_{N}},T_{\alpha}^{c}\ge t]
 &=
 \E[n_{\alpha}^{\ell}(t)\mid \Fs_{s-\tau_{N}},T_{\alpha}^{c}\ge t,T_{\alpha}^{d}\ge t]P(T_{\alpha}^{d}\ge t\mid \Fs_{s-\tau_{N}},T_{\alpha}^{c}\ge t)
 \\&+
 \E[n_{\alpha}^{\ell}(t)\mid \Fs_{s-\tau_{N}},T_{\alpha}^{c}\ge t, T_{\alpha}^{d}< t]P(T_{\alpha}^{d}< t\mid \Fs_{s-\tau_{N}},T_{\alpha}^{c}\ge t ).
 \end{align*}   
 Since $n_{\alpha}^{\ell}(t)=0$ if $T_{\alpha}^{d}< t$, we get that $
  \E[n_{\alpha}^{\ell}(t)\mid \Fs_{s-\tau_{N}},T_{\alpha}^{c}\ge t,T_{\alpha}^{d}\ge t] = \frac{\E[n_{\alpha}^{\ell}(t)\mid \Fs_{s-\tau_{N}},T_{\alpha}^{c}\ge t]}{P(T_{\alpha}^{d}\ge t\mid \Fs_{s-\tau_{N}},T_{\alpha}^{c}\ge t)}$. Thus, for a particle $\alpha$ alive at time $s-\tau_{N}$ we have 
\begin{align*}
    P(T_{\alpha}^{d}\ge t \mid \Fs_{s-\tau_{N}},T_{\alpha}^{c}\ge t,T_{\alpha}^{d}\ge s-\tau_{N})
    \ge
    P(T_{\alpha}^{d}\ge s \mid \Fs_{s-\tau_{N}},T_{\alpha}^{c}\ge t,T_{\alpha}^{d}\ge s-\tau_{N})
    = e^{-N\tau_{N}}
    > 0.5 ,
\end{align*}
where the last inequality follows by our choice of $N>N_{1}=4$. Thus,
\begin{align*}
    &\Xi\le
    \E\left(\sum_{\alpha\in \As(s-\tau_{N})} 
    \int_{s-\tau_{N}}^{s}2\lambda_{b}\sum_{\ell}\frac{p_{\ell}}{h_{\ell}}\E[{n_{\alpha}^{\ell}(t)} \mid \Fs_{s-\tau_{N}},T_{\alpha}^{c}\ge t]dt \right).&
\end{align*}
By the linearity of expectation and measurability of $\As(s-\tau_{N})$ with respect to $\Fs_{s-\tau_{N}}$ we get
\begin{align*}
    \Xi &\le
    \int_{s-\tau_{N}}^{s}2\lambda_{b} \sum_{\ell}\frac{p_{\ell}}{h_{\ell}} \E\left(\E\left[\sum_{\alpha\in \As(s-\tau_{N})} {n_{\alpha}^{\ell}(t)} \mid \Fs_{s-\tau_{N}},T_{\alpha}^{c}\ge t\right] \right)dt
    \le
    \int_{s-\tau_{N}}^{s}2\lambda_{b} \sum_{\ell}\frac{p_{\ell}}{h_{\ell}} \E\left(\sum_{\alpha\in \As(s-\tau_{N})} {n_{\alpha}^{\ell}(t)}  \right)dt.
\end{align*}
Recall that $\hat{I}_{t}^{l,N}$ denotes twice the expected number of $\ell$-neighbours alive at time $t$, and $\lambda_{b}=N+\theta$. Thus,
\begin{align*}
    \Xi &\le
    \sum_{\ell}\frac{p_{\ell}}{h_{\ell}}2(N+\theta)\int_{s-\tau_{N}}^{s}\hat{I}_{t}^{l,N}dt.
\end{align*}

Now we can use the bound on $\Xi$ to finish the proof of the lemma. Since $N>\theta$ we have
\begin{align*}
   \int_{\tau_{N}}^{t} N\E\left[\sum_{\alpha\in \As(s-\tau_{N})}\zeta_{\alpha}(s-\tau_{N})\right]ds
   &\le
    \int_{\tau_{N}}^{t} 2N(N+N)\sum_{\ell}\frac{p_{\ell}}{h_{\ell}}\int_{s-\tau_{N}}^{s}\hat{I}_{r}^{\ell,N} drds.
    \\&\le
    \sum_{\ell}\frac{p_{\ell}}{h_{\ell}}
    4N^{2}\tau_{N}
    \int_{0}^{t} \hat{I}_{r}^{\ell,N} dr,
 \end{align*}
 and we are done.
\end{proof}

We also need the following two lemmas.
\begin{lemma}\label{big}
For any $N>0$ and $t\ge 0$
\begin{equation*} 
  (1+\frac{\theta}{N})\sum_{\ell}{\frac{p_{\ell}}{h_l}\int_{0}^{t}N\hat{I}_{s}^{\ell,N}}ds\le 1+\int_{0}^{t} \theta\hat{m}_{s}^{N}ds.
\end{equation*}
\end{lemma}
\begin{proof}
From Lemma \ref{Differential equation} we immediately have
 \begin{equation*}   
  (1+\frac{\theta}{N})\sum_{\ell}{\frac{p_{\ell}}{h_l}\int_{0}^{t}N\hat{I}_{s}^{\ell,N}}ds = -\hat{m}_{t}^{N}+1+\int_{0}^{t} \theta\hat{m}_{s}^{N}ds \le
  1+\int_{0}^{t} \theta\hat{m}_{s}^{N}ds
\end{equation*}  
\end{proof}

\begin{lemma} \label{first m} For any $N>0$ and any $s\ge\tau_{N}$
\begin{equation*}
       \hat{m}_{s-\tau_{N}}^{N}\ge \hat{m}_{s}^{N}e^{-\theta\tau_{N}} .
\end{equation*}
\end{lemma} 
\begin{proof}
    From Lemma \ref{Differential equation} and since $\hat{m}_{0}^{N}=1$ we have
    \begin{align*}
    \hat{m}_{t}^{N}&=1+\int_{0}^{t} \theta\hat{m}_{s}^{N}ds - (1+\frac{\theta}{N})\sum_{\ell}{\frac{p_{\ell}}{h_l}\int_{0}^{t}N\hat{I}_{s}^{\ell,N}}ds
    \le 1+\int_{0}^{t} \theta\hat{m}_{s}^{N}ds.
    \end{align*}
    Thus, since $1$ is a non-decreasing function, and $\hat{m}_{r}^{N}<\infty$ for all $r\in[0,t]$ (see Lemma \ref{finiteM}) Gr{\" o}nwall's lemma immediately implies
    \begin{equation}\label{finitem}
      \hat{m}_{t}^{N}\le\hat{m}_{r}^{N}e^{\theta(t-r)} \;\forall r\in[0,t].
    \end{equation}
\end{proof}

In order to complete the proof of Proposition \ref{mainprop} we need the following proposition, which will be proven in Section \ref{vartheta}.
\begin{prop}
\label{varthetaprop}
For every $\epsilon>0$ there exists $N_{\epsilon}>0$ such that for every $N>N_{\epsilon}$
\begin{equation*}
      \vartheta-\epsilon\le\sum_{\ell}{\frac{p_{\ell}}{h_{\ell}}2N\E[Z_{1}^{\ell}] }.
    \end{equation*}
\end{prop}
Now we are ready to complete the proof of Proposition \ref{mainprop}.
\subsection{Proof of Proposition \ref{mainprop}}
In what follows we assume that $N>N_{2}$ where $N_{2}$ is defined in Lemma \ref{bound zeta integral}.
From Lemma \ref{Differential equation} we get
\begin{equation}\label{xcvkljbh}
    \hat{m}_{t}^{N}=\hat{m}_{\tau_{N}}^{N}+\int_{\tau_{N}}^{t} \theta\hat{m}_{s}^{N}ds- (1+\frac{\theta}{N})\sum_{\ell}{\frac{p_{\ell}}{h_l}\int_{\tau_{N}}^{t}N\hat{I}_{s}^{\ell,N}}ds .
\end{equation}
Apply on (\ref{xcvkljbh}) the bound on $\hat{I}_{s}^{\ell,N}$ that we have from Lemma \ref{bound I} to get
\begin{align*}
    \hat{m}_{t}^{N}&\le
    \hat{m}_{\tau_{N}}^{N}+\int_{\tau_{N}}^{t} \theta\hat{m}_{s}^{N}ds
    -
    (1+\frac{\theta}{N})\sum_{\ell}{\frac{p_{\ell}}{h_{\ell}}\int_{\tau_{N}}^{t} N \left (2\hat{m}_{s-\tau_{N}}^{N}\E[Z_{1}^{\ell}]-2\E\left[\sum_{\alpha\in \As(s-\tau_{N})}\zeta_{\alpha}(s-\tau_{N})\right]\right) }ds \nonumber.
\end{align*}
Apply the bound on $\int_{\tau_{N}}^{t} N\E\left[\sum_{\alpha\in \As(s-\tau_{N})}\zeta_{\alpha}(s-\tau_{N})\right]ds$ that we have from Lemma \ref{bound zeta integral} to get,
\begin{align*}
    \hat{m}_{t}^{N}\le
    \hat{m}_{\tau_{N}}^{N}+\int_{\tau_{N}}^{t} \theta\hat{m}_{s}^{N}ds
    &-
    (1+\frac{\theta}{N})\sum_{\ell}{\frac{p_{\ell}}{h_{\ell}}\int_{\tau_{N}}^{t} 2N \hat{m}_{s-\tau_{N}}^{N}\E[Z_{1}^{\ell}]}
   +
    2(1+\frac{\theta}{N})\sum_{\ell}{\frac{p_{\ell}}{h_l} 
     4N\tau_{N}
    \int_{0}^{t} N\hat{I}_{r}^{\ell,N} dr}.
\end{align*} 
Apply the bound on $(1+\frac{\theta}{N})\sum_{\ell}{\frac{p_{\ell}}{h_l}\int_{0}^{t}N\hat{I}_{s}^{\ell,N}}ds$ from Lemma \ref{big}, and the bound $\hat{m}_{\tau_{N}}^{N}\le e^{\theta\tau_{N}}$ from Lemma \ref{first m} to get
\begin{align*}
\hat{m}_{t}^{N}
    &\le
    \hat{m}_{\tau_{N}}^{N}+\int_{\tau_{N}}^{t} \theta\hat{m}_{s}^{N}ds
    -
    (1+\frac{\theta}{N})\sum_{\ell}{\frac{p_{\ell}}{h_{\ell}}\int_{\tau_{N}}^{t} 2N\hat{m}_{s-\tau_{N}}^{N}\E[Z_{1}^{\ell}]ds}
     +
    8N\tau_{N}(1+\int_{0}^{t} \theta\hat{m}_{s}^{N}ds).
    \\&\le
   \int_{\tau_{N}}^{t} \theta\hat{m}_{s}^{N}ds
    -
   \left ( \int_{\tau_{N}}^{t} \hat{m}_{s-\tau_{N}}^{N}ds\right)2N(1+\frac{\theta}{N})\sum_{\ell}{\frac{p_{\ell}}{h_{\ell}}\E[Z_{1}^{\ell}]}
    +e^{\theta\tau_{N}}+
    8N\tau_{N}(1+\int_{0}^{t} \theta\hat{m}_{s}^{N}ds).
\end{align*}
Apply the bound on $\hat{m}_{s-\tau_{N}}^{N}$ that we have from Lemma \ref{first m} to get
\begin{align*}
\hat{m}_{t}^{N}
    \le
   \int_{\tau_{N}}^{t} \theta\hat{m}_{s}^{N}ds
    &-
   \left (\int_{\tau_{N}}^{t} \hat{m}_{s}^{N}e^{-\theta\tau_{N}}ds\right)2N(1+\frac{\theta}{N})\sum_{\ell}{\frac{p_{\ell}}{h_{\ell}}\E[Z_{1}^{\ell}]}
    +e^{\theta\tau_{N}}+
    8N\tau_{N}(1+\int_{0}^{t} \theta\hat{m}_{s}^{N}ds).
\end{align*}
From (\ref{finitem}) in the proof of  Lemma \ref{first m} it is easy to see that $\hat{m}_{s}^{N}\le e^{\theta s}\le e^{\theta t}$ for any $s\in [0,t]$. Thus,
\begin{align*}
    \hat{m}_{t}^{N}\le
   \int_{\tau_{N}}^{t} \theta\hat{m}_{s}^{N}ds
    -
   e^{-\theta\tau_{N}}\left (\int_{\tau_{N}}^{t} \hat{m}_{s}^{N}ds\right)2N(1+\frac{\theta}{N})\sum_{\ell}{\frac{p_{\ell}}{h_{\ell}}\E[Z_{1}^{\ell}]}
    +e^{\theta\tau_{N}}+
    8N\tau_{N}(1+t\theta e^{\theta t}).
\end{align*}
Now, since $N\tau_{N}$ and $\tau_{N}$ approach zero as $N$ approaches infinity, we can find $N_{3}$ such that for every $N>N_{3}$ ,  $e^{\theta\tau_{N}}\le \frac{e}{2}$ and  $8N\tau_{N}(1+t\theta e^{\theta t})\le \frac{e}{2}$ for every $t\ge 0$.
Thus, for any $N>\max(N_{2},N_{3})$
\begin{align*}
    \hat{m}_{t}^{N}\le
   e+\left(\int_{\tau_{N}}^{t} \hat{m}_{s}^{N}ds\right)
   \left(\theta-e^{-\theta\tau_{N}} (1+\frac{\theta}{N})\sum_{\ell}{\frac{p_{\ell}}{h_{\ell}}2N\E[Z_{1}^{\ell}]}
   \right) .
\end{align*}
Use Proposition \ref{varthetaprop}, to see that for any $\epsilon>0$ there exists $N_{\epsilon}$ such that 
\begin{align*}
    \hat{m}_{t}^{N}&\le
   e+\left(\int_{\tau_{N}}^{t} \hat{m}_{s}^{N}ds\right)
   \left(\theta-e^{-\theta\tau_{N}} (1+\frac{\theta}{N})(\vartheta-\epsilon)
   \right),\;\;\;\;\;\; \forall N>\max(N_{2},N_{3},N_{\epsilon}).
\end{align*}
Again, by using Gr{\" o}nwall's lemma we get
\begin{align*}
    \hat{m}_{t}^{N}&\le
    e\cdot\exp\left[\left(\theta-e^{-\theta\tau_{N}} (1+\frac{\theta}{N})(\vartheta-\epsilon)
   \right)(t-\tau_{N}) \right].
\end{align*}
Since $\tau_{N}=\frac{\ln(N)}{N^2}$, there exists $N_{4}$ such that for any $N>N_{4}$ we have $e^{-\theta\tau_{N}} (1+\frac{\theta}{N})\ge 1$. Thus,
\begin{align*}
    \hat{m}_{t}^{N}&\le
    e\cdot \exp{\left [\left(\theta-1\cdot(\vartheta-\epsilon)\right)(t-\tau_{N})\right ]}
    =
    \exp{\left[\left(\theta-\vartheta+\epsilon\right)(t-\tau_{N})+1\right ]}, 
    \;\;\;\;\;\forall N>\max(N_{2},N_{3},N_{4},N_{\epsilon}).
\end{align*}
Recall that $\vartheta-\theta>0$, and fix $\epsilon:= \frac{\vartheta-\theta}{2}$ to get
\begin{align*}
    \hat{m}_{t}^{N}&\le
    \exp{\left [\frac{1}{2}\left(\theta-\vartheta\right)(t-\tau_{N})+1] \right]} 
    .
\end{align*}
Choose $N_{5}$ such that for any $N>N_{5}$ we have $\frac{1}{2}(\vartheta-\theta)\tau_{N}<1$. Then, 
\begin{align*}
    \hat{m}_{t}^{N}&\le
    \exp{\left [\frac{1}{2}\left(\theta-\vartheta\right)t+2 \right]},\;\;\;\;\;\; \forall {N>N_{\theta}:=\max(N_{2},N_{3},N_{4},N_{5},N_{\epsilon})}.
\end{align*}

\section{Proof of Proposition \ref{varthetaprop}}\label{vartheta}
This section is devoted for the proof of the Proposition \ref{varthetaprop}. In particular we show that
\begin{equation*}
      \vartheta= \liminf_{N\to\infty}\sum_{\ell}{\frac{p_{\ell}}{h_{\ell}}2N\E[Z_{1}^{\ell}] },
    \end{equation*}
which is enough to prove the proposition.  
Let us first introduce an additional notation: 
 \begin{itemize}  

     \item Let $e_{i}$ denote the vector in $\Zd$ with 1 in the $i$-th coordinate and $0$'s elsewhere.
     \item  For $\ell\ge 1$, let $\varphi_{d}^{\ell}(x)= \{y\in \Zd :{\|y-x\|}_1=l  \} \setminus\{0\}$ denote the $\ell$-neighbourhood of $x$ excluding the origin.
     \item For $\ell\ge 1$, let $\varphi_{d}^{\ell}=\varphi_{d}^{\ell}(0)$ denote the $\ell$-neighbourhood of the origin.
     \item  Let $\varphi_{d}(x)= \varphi_{d}^{1}(x)$ denote the nearest neighbours of $x$, and recall that $\varphi_{d}=\varphi_{d}(0)$.
     \item  Let $A_{d}(x)= \{y : y\in\varphi_{d}\;,\;y\in\varphi_{d}(x)\}$ denote the set of points which belong both to the nearest neighbours of the origin and to the nearest neighbours of $x$.
 \end{itemize}
 In addition, recall that for any $x\in\Zd$, $x[i]$ denotes the $i$-th coordinate in the vector $x$.\\
 
 Before we give the proof of Proposition \ref{varthetaprop} let us state a few auxiliary lemmas.
 
\begin{lemma}\label{A}
For each $x\in\varphi_{d}$ there is exactly one point (which we denote by $z_x$) in $\varphi_{d}(x)$ such that $|A_{d}(z_x)|=1$. For any other point, $w_{x}\in\varphi_{d}(x)$, the following holds: $|A_{d}(w_{x})|=2$.
\end{lemma}
\begin{proof}
Take $x\in\varphi_{d}$. Without loss of generality assume $x=e_{i}$, for some $i=1,2,...,d$. 
 \\Take $z_{x}=e_{i}+e_{i}$. Then
    $z_{x}\in\varphi_{d}(x)$ and $A_{d}(z_{x})=\{e_{i}\} \Rightarrow |A_{d}(z_{x})|=1$.\\
 Now, take $y\in\varphi_{d}(x)$ such that $y\neq z_{x}$. Then
 $y[i]=1$ and there exists only one ${j\in\{1,2,3,...d\}\backslash \{i\}}$ such that $y[j]=1$ or $y[j]=-1$. Thus, $A_{d}(y)=\{e_{i},e_{j}\}$ (or $A_{d}(y)=\{e_{i},-e_{j}\}$) and the result follows. 
\end{proof}
In order to express $\E(Z_{1}^{\ell})$ in terms of Green's function we are going to couple the stirring process with a simple symmetric random walk. To this end we define the following:
\begin{definition}
Let $\{V_{t}^{N}\}_{t\ge 0}$ be a continuous time symmetric nearest neighbour random walk on $\Zd$ with jump rate $4dN^{2}$ and let $Q^{V}$ be its transition rate matrix. Let $\{W_{t}^{N}\}_{t\ge 0}$ be a continuous time Markov chain taking values in $\Zd$ whose transition rate matrix is defined as follows:
    \begin{equation*}
    Q^{W}(x,y)=
    \fourpartdef
    {Q^{V}(x,y)}  {x\notin\varphi_{d},}
    {N^{2}}    {x\in\varphi_{d}\;\;\text{and}\;\; y=-x,}
    {2N^{2}} {x\in\varphi_{d}\;\;\text{and}\;\;y\in\varphi_{d}(x)}
    {0}    {\text{otherwise.}}
    \end{equation*}
\end{definition}
In order to prove Proposition \ref{varthetaprop} we would like to show that the expected times $\{V_{t}^{N}\}_{t\ge 0}$ and $\{W_{t}^{N}\}_{t\ge 0}$ spend in $\varphi_{d}$ are the same. To this end, we partition the points that the processes can exit $\varphi_{d}\cup \{0\}$ via them. Notice that when $\{V_{t}^{N}\}_{t\ge 0}$ or $\{W_{t}^{N}\}_{t\ge 0}$ leaves $\varphi_{d}\cup\{0\}$ it does it via some point in $\varphi_{d}^{2}(0)$. Denote
\begin{equation*}
        J_{1}=\bigcup_{x\in\varphi_{d}^{2}(0)\;:\;|A_{d}(x)|=1}\{x\},\;\;\;\;\; \;\;\;\;\;
    J_{2}=\bigcup_{x\in\varphi_{d}^{2}(0)\;:\;|A_{d}(x)|=2}\{x\}.
\end{equation*}
From Lemma \ref{A} it follows that for any $x\in\varphi_{d}^{2}(0)$ either $|A_{d}(x)|=1$ or {$|A_{d}(x)|=2$}.  \\   
In the following lemma we prove that the probability that the random walk $\{V_{t}^{N}\}_{t\ge 0}$ exits the neighbourhood of the origin via points in $J_{1}$ is equal to the probability that $\{W_{t}^{N}\}_{t\ge 0}$ exits the neighbourhood of the origin via points in $J_{1}$.
\begin{lemma}\label{first_step_cond}
 Let $V_{0}^{N},W_{0}^{N}\in\varphi_{d}$. 
 Let $\tau_{V}$ and $\tau_{W}$ be the time that $\{V_{t}^{N}\}_{t\ge 0}$ and $\{W_{t}^{N}\}_{t\ge 0}$ leave $\varphi_{d}\cup\{0\}$ for the first time respectively . 
 Then,
 \begin{align*}
     (1)\;\;\;\;P(V_{\tau_{V}}\in J_{1})=P(W_{\tau_{W}}\in J_{1}),\;\;\;\;\; \;\;\;\;\;\;\;\;
     (2)\;\;\;\;P(W_{\tau_{V}}\in J_{2})=P(W_{\tau_{W}}\in J_{2}).
 \end{align*}
\end{lemma}

\begin{proof}
As mentioned above, from Lemma \ref{A} it follows that
\begin{equation}\label{J}
        J_{1}\cap J_{2}=\emptyset\;\;,\;\; J_{1}\cup J_{2}=\varphi_{d}^{2}.
\end{equation}
In addition, notice that $\tau_{V}$ and $\tau_{W}$ are finite stopping times.
Thus, (2) follows immediately from (1).\\
Assume $V_{0}^{N}\in\varphi_{d}$. Thus, using the first step analysis we obtain:
\begin{equation*}
    P(V_{\tau_{V}}\in J_{1})=\frac{1}{2d}+0\cdot\frac{2d-2}{2d}+\frac{1}{2d}P(V_{\tau_{V}}\in J_{1}) \Rightarrow P(V_{\tau_{V}}\in J_{1})=\frac{1}{2d-1}.
\end{equation*}
This holds since when $\{V_{t}^{N}\}_{t\ge 0}$ jumps from $V_{0}^{N}$, its location after the jump is uniformly distributed within the closest neighbours of $V_{0}^{N}$. Thus, with probability $\frac{1}{2d}$ it jumps to a point in $J_{1}$; with probability $\frac{2d-2}{2d}$ it jumps to a point in $J_{2}$ and with probability $\frac{1}{2d}$ it jumps to the origin (and then jumps back to $\varphi_{d}$).\\

Now consider $\{W_{t}^{N}\}_{t\ge 0}$. 
Assume $W_{0}^{N}=x\in\varphi_{d}$. Thus, again using the first step and (\ref{J}) we obtain:
\begin{equation*}
    P(W_{\tau_{V}}\in J_{1})=\frac{2}{4d-1}+0\cdot\frac{2(2d-2)}{4d-1}+\frac{1}{4d-1}P(W_{\tau_{V}}\in J_{1}) \Rightarrow P(W_{\tau_{V}}\in J_{1})=\frac{1}{2d-1}.
\end{equation*}
This holds since when $\{W_{t}^{N}\}_{t\ge 0}$ jumps from $x$, with rate $2N^{2}$ it jumps to $z_{x}\in J_{1}$;
with rate $N^{2}$ it jumps to $-W_{0}^{N}\in\varphi_{d}$
and with rate ${2N^{2}(2d-2)}$ it jumps to a point in $J_{2}$.


\end{proof}
 
In the next lemma we show that the expected times $V^{N}$ and $W^{N}$ spend in $\varphi_{d}^{\ell}$ are the same.
\begin{lemma}\label{W and V}
For every $l\in\{1,2,3,...\} $,
 \begin{equation*}
    \E\left[\int_{0}^{t}1_{\{V_{s}^{N}\in \varphi_{d}^{\ell}\}}ds\right]= \E\left[\int_{0}^{t}1_{\{W_{s}^{N}\in \varphi_{d}^{\ell}\}}ds \right].
 \end{equation*}
\end{lemma}
\begin{proof}
It is shown in the proof of Lemma 3.4 in \cite{berezin2014asymptotic} that once $\{V_{t}^{N}\}_{t\ge 0}$ enters $\varphi_{d}$ the time $\{V_{t}^{N}\}_{t\ge 0}$ spends in $\varphi_{d}$ before leaving the set $\varphi_{d}\cup\{0\}$ has the same distribution as the time spent by $\{W_{t}^{N}\}_{t\ge 0}$ in $\varphi_{d}$ at any visit of this set.\\
By Lemma \ref{first_step_cond}, the strong Markov property and the fact that, in distribution, the behavior of these two processes is exactly the same outside of $\varphi_{d}$ the result follows.
\end{proof}
Let us state one more lemma (it is standard so its proof is omitted).
\begin{lemma}\label{discVScontRW}
Let $\{V_{s}\}_{s>0}$ be a continuous time symmetric random walk on $\Zd$ with jump rate $\lambda$ started at $V_{0}=0$. Let $\{D_{n}\}_{n>0}$ be a discrete time symmetric random walk on $\Zd$ started at $D_{0}=0$. Let $\pi(u)$ be a Poisson process with rate $1$ independent of $\{D_{n}\}_{n>0}$. Then, for any $l=1,2,...$ and for any $t>0$
\begin{equation*}
  \int_{0}^{t}{1_{\{V_{s}\in \varphi_{d}^{\ell}\}}ds}=
  \int_{0}^{t}{1_{\{D_{\pi(\lambda s)}\in \varphi_{d}^{\ell}\}}ds}.
\end{equation*}
\end{lemma}

\medskip\noindent
{We are now ready to give the proof of Proposition \ref{varthetaprop}}.
\begin{proof}[Proof of Proposition \ref{varthetaprop}\nopunct].
Recall that we start the process $\{\xi_{t}^{N}\}_{t\ge 0}$ with one particle denoted as particle $1$. If particle $1$ splits, $(1,0)$ and $(1,1)$ denote the two children of particle $1$.
Let $F_{1}=\{T_{1}^{c}<T_{1}^{d},T_{1}^{c}<\tau_{N},T_{(1,0)}>\tau_{N},T_{(1,1)>\tau_{N}\}}$ denote the event that particle $1$ splits before time $\tau_{N}$ and its two children are alive at time  $\tau_{N}$. In addition recall that $\tau_{N}=\frac{\ln(N)}{N^{2}}$. Then, we have
\begin{equation}
    P(F_{1})= \frac{N+\theta}{2N+\theta}e^{-(2N+\theta)\tau_{N}}(1-e^{-(2N+\theta)\tau_{N}}).
\end{equation}
Now, fix arbitrary $\ell$. Notice that $E(Z_{1}^{\ell}|F_{1}^{c})=0$. Thus, we only care about
    $P(B_{\tau_{N}}^{(1,0)}\underset{\ell}{\sim}B_{\tau_{N}}^{(1,1)}|F_{1})$.
Given $F_{1}$, the time at which particle 1 splits is uniformly distributed in $[0,\tau_{N}]$. Thus,
\begin{equation}\label{momm}
     E(Z_{1}^{\ell}|F_{1})=\frac{1}{\tau_{N}}\int_{0}^{\tau_{N}}P(L+W^{N}_{\tau_{N}-t}\in\varphi_{d}^{\ell})dt=\frac{1}{\tau_{N}}\int_{0}^{\tau_{N}}P(L+W_{s}^{N}\in\varphi_{d}^{\ell})ds
\end{equation}
where $L$ is the difference in position of the two children of particle 1 right after its split.\\
Apply Lemma \ref{W and V} on the right-hand side of (\ref{momm}) to get
\begin{equation*}
       E(Z_{1}^{\ell}|F_{1})=\frac{1}{\tau_{N}}\int_{0}^{\tau_{N}}P(L+V_{s}^{N}\in\varphi_{d}^{\ell})ds.
\end{equation*}
Now, let $\{D_{n}\}_{n\ge 1}$ be a simple (discrete) symmetric random walk on $\Zd$ independent of $\ell$ and started at $D_{0}=0$.  Let $\{\pi(s)\}_{s>0}$ be the Poisson process with rate $1$ defined on the same probability space and independent of $\{D_{n}\}_{n\ge 1}$ and $L$. Thus, since $\{V_{s}^{N}\}_{s\ge 0}$ is a continuous time symmetric random walk with jump rate $4dN^{2}$, by Lemma \ref{discVScontRW} we get
\begin{align*}
    E(Z_{1}^{\ell}|F_{1})=\frac{1}{\tau_{N}}\int_{0}^{\tau_{N}}P(L+D_{\pi(4dN^{2}s)}\in\varphi_{d}^{\ell})ds
=
      \frac{1}{\tau_{N}}\int_{0}^{4dN^{2}\tau_{N}}(4dN^{2})^{-1}P(L+D_{\pi(r)}\in\varphi_{d}^{\ell})dr.
\end{align*}
Overall we have
\begin{align}\label{Z}
    E(Z_{1}^{\ell}) =\frac{N+\theta}{2N+\theta}e^{-(2N+\theta)\tau_{N}}(1-e^{-(2N+\theta)\tau_{N}})\frac{1}{\tau_{N}}(4dN^{2})^{-1}\int_{0}^{4dN^{2}\tau_{N}}P(L+D_{\pi(r)}\in\varphi_{d}^{\ell})dr.
\end{align}
Taking $N$ to infinity, we get that
for every $\ell\ge 1$
\begin{align}\label{everyl}
    \lim_{N\to\infty}{NE(Z_{1}^{\ell})}&=(4d)^{-1}\int_{0}^{\infty}P(L+D_{\pi(r)}\in\varphi_{d}^{\ell})dr
    \nonumber\\&=
    (4d)^{-1}\sum_{n=0}^{\infty}P(L+D_{n}\in\varphi_{d}^{\ell}).
\end{align}
where the second equality follows since the times between jumps of $D_{\pi(r)}$ are exponential with mean $1$.\\
Denote \begin{equation}\label{Theta}
    \Theta=\liminf_{N\to\infty}\sum_{\ell}{\frac{p_{\ell}}{h_{\ell}}2N\E[Z_{1}^{\ell}]}.
    \end{equation}
From (\ref{Z}), it is easy to check that $N\E[Z_{1}^{\ell}]\le (1+\theta)\sum_{n=0}^{\infty}P(L+D_{n}\in \varphi_{d}^{\ell})$ for all $N\ge 1$. Thus, by the Dominated Convergence theorem we get
    \begin{equation*}
    \Theta=
    2\sum_{\ell}{\frac{p_{\ell}}{h_{\ell}}\liminf_{N\to\infty}N\E[Z_{1}^{\ell}]}.
\end{equation*}
Applying (\ref{everyl}) on the right-hand side of the last term we get
\begin{align*}
\Theta=
    2\sum_{\ell}\frac{p_{\ell}}{h_{\ell}}(4d)^{-1}\sum_{n=0}^{\infty}P(L+D_{n}\in\varphi_{d}^{\ell})
    =
    (2d)^{-1}\sum_{\ell}\frac{p_{\ell}}{h_{\ell}}\sum_{n=0}^{\infty}\sum_{y\in\varphi_{d}^{\ell} }P(L+D_{n}= y).
\end{align*}
By Tonelli's theorem for non-negative functions and the total probability law we get
\begin{align}\label{dadd}
    \Theta=&(2d)^{-1}\sum_{\ell}\frac{p_{\ell}}{h_{\ell}}\sum_{y\in\varphi_{d}^{\ell}}\sum_{n=0}^{\infty}P(L+D_{n}= 
    (2d)^{-1}\sum_{\ell}\frac{p_{\ell}}{h_{\ell}}\sum_{y\in\varphi_{d}^{\ell}}\sum_{n=0}^{\infty}\sum_{x\in \Zd}P(x+D_{n}= y)P(L=x).
\end{align}
Let $P^{n}(x,y)=P(D_{n}=y|D_{0}=x)$ be the probability that a discrete time symmetric random walk on $\Zd$, started at site $x$, lands on site $y$ after $n$ steps. Recall also that $P(L=x)=P^{b}(0,x)$.  Thus, we get 
\begin{equation*}
    \Theta=
    (2d)^{-1}\sum_{\ell}\frac{p_{\ell}}{h_{\ell}}\sum_{y\in\varphi_{d}^{\ell}}\sum_{n=0}^{\infty}\sum_{x\in \Zd}P^{n}(x,y)P^{b}(0,x).
\end{equation*}
By Tonelli's theorem for non-negative functions, and the definition of Green's function we get
\begin{align*}
    \Theta&=(2d)^{-1}\sum_{x\in \Zd}P^{b}(0,x)\sum_{\ell}\sum_{y\in\varphi_{d}^{\ell}}\frac{p_{\ell}}{h_{\ell}}G(x,y).
\end{align*}
By the definition of $P^{b}(0,y)$ we have
\begin{equation*}
     \Theta =
    (2d)^{-1}\sum_{x\in \Zd}\sum_{y\in \Zd}P^{b}(0,x)P^{b}(0,y)G(x,y),
\end{equation*}
and we are done.
\end{proof}

\bibliographystyle{plain}
\bibliography{ref}

\end{document}